\documentclass{amsart}

\usepackage{amsmath}
\usepackage{verbatim}
\usepackage{amssymb}
\usepackage{graphicx}
\usepackage{hyperref}
\usepackage{color}

\numberwithin{equation}{section}

\newtheorem{theorem}{Theorem}
\newtheorem{lemma}{Lemma}

\theoremstyle{definition}
\newtheorem{rem}{Remark}
\newtheorem{assumption}{Assumption}

\newcommand{\dd}{\mathrm d}
\newcommand{\ee}{\mathrm e}

\newcommand{\cC}{\mathcal C}
\newcommand{\cD}{\mathcal D}
\newcommand{\cP}{\mathcal P}
\newcommand{\bP}{\mathbb P}
\newcommand{\cF}{\mathcal F}
\newcommand{\cL}{\mathcal L}
\newcommand{\IR}{\mathbb R}
\newcommand{\IN}{\mathbb N}
\newcommand{\vp}{\varphi}
\newcommand{\HS}{\mathrm{HS}}
\newcommand{\QQ}{B}

\DeclareMathOperator{\Span}{span}

\DeclareMathOperator{\dom}{dom}
\DeclareMathOperator{\Tr}{Tr}

\title[MLEI for a fractional order equation]
{Mittag-Leffler Euler integrator for a 
stochastic fractional order equation 
with additive noise}

\author[M.~Kov\'acs]{Mih\'aly Kov\'acs}

\address{
  Department of Mathematical Sciences,
  Chalmers University of Technology and University 
  of Gothenburg,
  SE--412 96 Gothenburg,
  Sweden and 
  P\'azm\'any P\'eter Catholic University, 
  Faculty of Information Technology an Bionics, P.O. Box 278, H-1444 Budapest, Hungary}
\email{mihaly@chalmers.se}

\author[S.~Larsson]{Stig Larsson}

\address{
  Department of Mathematical Sciences, 
  Chalmers University of Technology and University 
  of Gothenburg,
  SE--412 96 Gothenburg,
  Sweden}

\email{stig@chalmers.se}

\author[F.~Saedpanah]{Fardin Saedpanah}

\address{
  Department of Mathematics,
  University of Kurdistan,
  P.O. Box 416, Sanandaj,
  Iran}

\email{f.saedpanah@uok.ac.ir, fardin.saedpanah@gmail.com}

\thanks{Research supported in part by 
Nordforsk, project number 74756.}

\keywords{Euler integrator, fractional equations, Riesz kernel, strong
  convergence, integro-differential equations, stochastic differential
  equations.}

\subjclass{34A08, 45D05, 45K05, 60H15, 60H35, 65M12, 65M60}

\begin{document}

\begin{abstract}
  Motivated by fractional derivative models in viscoelasticity, a
  class of semilinear stochastic Volterra integro-differential
  equations, and their deterministic counterparts, are considered.  A
  generalized exponential Euler method, named here as the
  Mittag-Leffler Euler integrator, is used for the temporal
  discretization, while the spatial discretization is performed by the
  spectral Galerkin method.  The temporal rate of strong convergence
  is found to be (almost) twice compared to when the backward Euler
  method is used together with a convolution quadrature for time
  discretization.  Numerical experiments that validate the theory are
  presented.
\end{abstract}

\date{\today} 
  
\maketitle

\section{Introduction}  \label{sec:intro}

We study the numerical approximation of a class of semilinear
Volterra integro-differential equations in a real, separable,
infinite-dimensional Hilbert space $H$ of the form
\begin{equation} \label{VolterraEq}
  \dd u(t)+\int_0^t b(t-s)Au(s)\, \dd s \, \dd t
  =F(u(t))\, \dd t+\dd W(t),
    \, t\in (0,T];\,u(0)=u_0, 
\end{equation}
where $A$ is a self-adjoint, positive definite, not necessarily
bounded, operator on the Hilbert space $H$, $W$ is an $H$-valued
Wiener process with covariance operator $Q$, $F\colon H\to H$ is a
nonlinear operator, $b$ is a locally integrable scalar kernel, and
$u_0$ is an $H$-valued random variable. Our main example of $b$ is the
Riesz kernel
\begin{equation} \label{RieszKernel}
 b(t)=\frac{t^{\alpha-1}}{\Gamma(\alpha)}, 
  \quad 0<\alpha<1,  
\end{equation} 
where 
$\Gamma(\alpha)=\int_0^\infty t^{\alpha-1}\ee^{-t}\,\dd t$ 
is the gamma function. 

By introducing the fractional integral of order $\alpha$ denoted by
$J_0^{\alpha}$, see, for example, \cite{PodlubnyBook1999}, as
\begin{equation*}
  (J_0^{\alpha}g)(t)=\frac{1}{\Gamma(\alpha)}
    \int_0^t (t-s)^{\alpha-1} g(s)\, \dd s,
     \quad 0<\alpha<1,
\end{equation*}
equation \eqref{VolterraEq} becomes a fractional order equation of the form
\begin{equation} \label{Volterrafrac}
  \dd u(t)+J_0^{\alpha}(Au)(t) \, \dd t
  =F(u(t))\, \dd t+\dd W(t) ,
  \  t\in (0,T];\quad u(0)=u_0. 
\end{equation}

We note that the present framework applies also to slightly more 
general kernels, which have similar smoothing effects, 
such as the tempered Riesz kernel
\begin{equation} \label{TemperedRieszKernel}
  b(t)=\frac{1}{\Gamma(\alpha)}t^{\alpha-1} \ee^{-\eta t},
  \quad 0<\alpha<1,\ \eta\geq 0,
\end{equation}
and even to certain kernels with finite smoothness, see
Remark~\ref{rem:S4} for further discussion. 
Recalling that
  $0<\alpha<1$ in \eqref{RieszKernel}, we denote henceforth
\begin{align} \label{def:rho}
 \rho=\alpha+1, \quad 1<\rho<2,
\end{align}
so that $b(t)=t^{\rho-2}/\Gamma(\rho-1)$.  

We motivate our main example \eqref{Volterrafrac} by a model from
linear viscoelasticity, for more examples see, e.g.,
\cite{Mainardi,PrussBook1993} and references therein. In one spatial
dimension, considering the class of viscoelastic materials which
exhibit a simple power-law creep, the evolution equation that
describes the response variable $w$ (chosen among the field variables:
displacement, stress, the strain or the particle velocity) is given by
\begin{equation*}
 w(x,t)=w(x,0+)+tw_t(x,0+)+ \frac{c}{\Gamma(\rho)}
\int_0^t (t-s)^{\rho-1} w_{xx}(x,s) \,\dd s,
  \quad 1<\rho<2,
\end{equation*}
see, for example, \cite{Mainardi}. Assuming that 
$w_t(x,0+)=0$ and that $w$ is continuous at $t=0+$ 
with $w(x,0+)=w_0(x)$, one arrives at the Cauchy problem 
\begin{equation*}
 w_t(x,t)= \frac{c}{\Gamma(\rho-1)}\int_0^t
   (t-s)^{\rho-2}w_{xx}(x,s)\,\dd s;
   \quad w(x,0)=w_0(x).
\end{equation*}
With $\alpha=\rho-1$ and $c=1$ we get 
\begin{equation*}
 w_t(x,t)=J^\alpha_0(w_{xx}(x,\cdot))(t); \quad w(x,0)=w_0(x).
\end{equation*}
Now, if $w$ is chosen to be the particle velocity, and $f$ represents
a nonlinear, external viscous force, which is perturbed by Gaussian
noise $\dot{\xi}$, then the equation for the particle velocity reads
as
\begin{equation}\label{eq:frw}
 w_t(x,t)=J^\alpha_0(w_{xx}(x,\cdot))(t)+f(w(x,t))
  +\dot{\xi}(x,t);\quad w(x,0)=w_0(x).
\end{equation}
Considering the equation on an interval $[0,L]$ and supplementing the
equation with non-slip boundary conditions, we arrive at a special
instance of \eqref{Volterrafrac}, with $H=L_2(0,L)$, $A$ being the
Dirichlet Laplacian in $H$ and $F$ the Nemytskij operator
$F(v)(\cdot)=f(v(\cdot))$.  We remark that, without the noise and $f$,
equation \eqref{eq:frw} is often referred to as a fractional wave
equation.

We note that when the kernel $b$ in \eqref{VolterraEq} is smooth, 
e.g., exponential kernels, these equations reveal a 
hyperbolic behaviour, whereas for weakly singular 
kernels, e.g., the Riesz kernel \eqref{RieszKernel}, 
they exhibit certain parabolic features. 

The literature on numerical methods for 
stochastic PDEs, such as stochastic parabolic 
and hyperbolic PDEs, is mature.
In some works, by using exponential integrators 
\cite{HochbruckOstermann}, the strong rate of 
convergence has been improved for the stochastic 
heat equation, see, e.g., 
\cite{BeckerJentzenKloeden,LordTambue,WangQi}, 
and for the stochastic wave equation, see, e.g., 
\cite{AntonCohenLarssonWang}, and the 
references therein. 
The drawback of the exponential integrators for 
stochastic PDEs is that, the eigenfunctions of the 
operator $A$ and of the covariance operator $Q$ 
of the noise must coincide and must be known 
explicitly, so that the scheme can be implemented. 

However, the literature on numerical analysis of stochastic Volterra
equations is more scarce, containing only
\cite{AnderssonKovacsLarsson,KovacsPrintemsStrong,KovacsPrintemsWeak},
and recently a few papers specifically for the fractional stochastic
heat equation (where there is a derivative in front of $J_0^{\alpha}$
in \eqref{Volterrafrac}), based on a convolution quadrature, see, for
example, \cite{Gunz,Gunz1}.

Here, we study a full discretization of \eqref{VolterraEq}, as well as
its deterministic counterpart, i.e., the special case when $W=0$.  We
use a generalized exponential Euler method, named here as the
Mittag-Leffler Euler integrator, for the temporal discretization. Full
discretization is then formulated by the spectral Galerkin method for
spatial discretization.

As is the case for stochastic equations with no memory effects, the
time integration is based on the mild formulation of the
equation. However, there is a major difference, namely, the solution
operator in our case do not have the semigroup property and hence the
integrator uses the approximate solution from all previous time
levels, not just from the current one, in order to advance to the next
time level. This phenomenon is of course present in the convolution
quadrature setting as well and makes the error analysis more difficult
compared to the memoryless case.

The main novelty in this work is the introduction of a new temporal
discretization method for \eqref{VolterraEq} and its error
analysis. The analysis of the spatial discretisation is more or less
standard.  In particular, we prove that the strong rate of temporal
convergence, is (almost) twice the rate of the Euler method combined
with Lubich's convolution quadrature of order 1,
\cite{AnderssonKovacsLarsson}. As a consequence, for trace-class
noise, we recover (almost) the optimal rate 1 in time.

When $H=L_2(\cD)$, where $\cD\subset \mathbb{R}^d$ is a bounded
domain, with appropriately smooth boundary, the framework presented
here allows for a general class of Nemytskij operators $F$ when
$d=1,2,3$, with some restriction on $\rho$ when $d=3$.  For
space-time white noise we must have $d=1$, while for coloured noise
$d>1$ is allowed.

The outline of the paper is as follows.  In Section~\ref{2}, 
we introduce notation, the abstract framework, state our main 
assumptions and present some preliminary results on the 
solution of \eqref{VolterraEq}. 
In Section~\ref{Sec3}, we introduce the numerical
scheme \eqref{UmN} and, in Theorem~\ref{thm:main}, we state 
and prove our main result on the order of strong convergence. 
In Section~\ref{Sec:Numeric}, we discuss the implementation of 
the scheme and present some numerical experiments to illustrate 
the theory. 
Throughout the paper $C$ denotes a generic constant that 
may have different values at different occurrences, but 
its value is independent of the discretization parameters. 

\section{Preliminaries} \label{2}
\subsection{The abstract setting}
Let $H$ be a real, separable, infinite-dimensional Hilbert space with
inner product $(\cdot,\cdot)$ and norm $\|\cdot\|$ and $A$ be a
self-adjoint, positive definite, not necessarily bounded operator in
$H$ with compact inverse.  An important example is $H=L_2(\cD)$ and
$A=-\Delta$ with homogeneous Dirichlet boundary conditions.  Let
$\{(\lambda_k,\vp_k)\}_{k=1}^\infty$ be the eigenpairs of $A$, i.e.,
\begin{equation} \label{EigenEq}
 A\vp_k=\lambda_k \vp_k,\quad k \in \IN.
\end{equation}
It is known that
$ 0<\lambda_1\leq \lambda_2\leq \dots \leq \lambda_k \leq \dots$
with $\lim_{k\to\infty}\lambda_k=\infty$ and the eigenvectors
$\{\vp_k\}_{k=1}^\infty$ form an orthonormal basis for $H$.  We
introduce the fractional order spaces
\begin{equation*}
 \dot H^\nu
 :=\dom(A^{\frac{\nu}{2}}), \quad
 \|v\|_\nu^2 :=\|A^{\frac{\nu}{2}} v\|^2
 =\sum_{k=1}^\infty \lambda_k^\nu (v,\vp_k)^2,\quad 
 \nu\in \IR,\ v \in \dot H^\nu.
\end{equation*}
 
Let $\cL=\cL(H)$ denote the space of all bounded 
linear operators on $H$. We also consider the space 
of Hilbert--Schmidt operators, that is, the space 
of all operators $T\in\cL$ for which 
\begin{equation*} 
 \|T\|_{\HS}^2 
  = \sum_{k=1}^\infty \|T\vp_k\|^2 < \infty. 
\end{equation*}

Let $(\Omega,\cF,(\cF_t)_{t\in [0,T]},\bP)$ be a filtered probability
space, $(\cF_t)_{t\in [0,T]}$ being a normal filtration, with Bochner
spaces $L_p(\Omega;H)=L_p\big((\Omega,\cF,\bP);H\big), \ p\geq 2$.  We
let $Q\in \cL$ be a self-adjoint, positive semidefinite operator and
$H_0=Q^{\frac{1}{2}}(H)$ be the Hilbert space with the inner product
$\langle u,v \rangle_{H_0}=\langle
Q^{-\frac{1}{2}}u,Q^{-\frac{1}{2}}v\rangle$, where $Q^{-\frac{1}{2}}$
denotes the pseudoinverse of $Q^{\frac{1}{2}}$, when it is not
injective, and $Q^{\frac{1}{2}}$ is the unique positive semidefinite
square root of $Q$.  By $\cL_2^0=\cL_2^0(H)$ we denote the space of
Hilbert--Schmidt operators $H_0\to H$. Thus,
$\| T\|_{\cL_2^0}=\| T Q^{\frac{1}{2}}\|_{\HS}<\infty$, for
$T\in \cL_2^0$.  Then we let $W$ be $Q$-Wiener process in $H$ with
respect to $(\Omega,\cF,(\cF_t)_{t\in [0,T]},\bP)$.  We recall the
It\^o isometry,
\begin{equation} \label{ItoIsometry}
  \Big\| \int_0^t \phi(s)\, \dd W(s)
   \Big\|_{L_2(\Omega;H)}
  =
\Big\|\Big(\int_0^t \big\|\phi(s)
   \big\|_{\cL_2^0}^2\, \dd s\Big)^{\frac{1}{2}}
    \Big\|_{L_2(\Omega;\IR)},
\end{equation}
and the Burkholder--Davis--Gundy inequality, 
for $p\geq 2$,   
\begin{equation} \label{BDG}
  \Big\| \int_0^t \phi(s)\, \dd W(s)
   \Big\|_{L_p(\Omega;H)}  
  \leq C_p\Big\|\Big(\int_0^t \big\|\phi(s)
   \big\|_{\cL_2^0}^2\, \dd s\Big)^{\frac{1}{2}}
    \Big\|_{L_p(\Omega;\IR)}, 
\end{equation}
for strongly measurable functions 
$\phi\colon[0,T]\to \cL_2^0$, \cite{DaPratoZabczyk}.

We recall from \eqref{def:rho} that $ \rho=\alpha+1$,
  $\rho\in (1,2),\ \alpha\in (0,1)$. 

\begin{assumption} \label{Assumption_W}
We quantify the regularity of the noise by 
$\beta\in(0,\frac{1}{\rho}]$ through the assumption 
that there is a constant $\QQ$ such that
\begin{equation} \label{NoiseRegularity}
  \big\| A^{\frac{\beta-\frac{1}{\rho}}{2}} 
   \big\|_{\cL_2^0}  
  =\big\| A^{\frac{\beta-\frac{1}{\rho}}{2}} 
   Q^{\frac{1}{2}}\big\|_{\HS}  \le\QQ.
\end{equation}
\end{assumption}
Trace class noise; that is, when 
$\Tr(Q)=\|Q^{\frac{1}{2}}\|_{\HS}^2<\infty$,
corresponds to $\beta=\frac{1}{\rho}$. 
When $A=-\Delta$ is the Dirichlet Laplacian, we may 
take $Q=A^{-s}$ with $s\geq 0$. Then 
\eqref{NoiseRegularity} is satisfied with 
$\beta<s+\frac{1}{\rho}-\frac{d}{2}$, because 
$\lambda_j \approx j^{\frac{2}{d}}$ as $j\to \infty$. 
We note that $s=0$ corresponds to 
space-time white noise $Q=I$, and in this case, $d=1$ and $\beta<\frac{1}{\rho}-\frac{1}{2}$. 

\subsection{The linear deterministic problem}
We assume that there exists a strongly continuous family
$\{S(t)\}_{t\ge 0}$ of bounded linear operators on $H$ such that the
function $u(t)=S(t)u_0$, $u_0\in H$, is the unique solution of
\begin{equation*}
u(t)+A\int_0^tB(t-s)u(s)\,\dd s=u_0,\quad t\ge 0,  
\end{equation*}
with $B(t)=\int_0^tb(s)\,\dd s$. When $t\to u(t)=S(t)u_0$ is differentiable for $t>0$, then $u$ is the unique solution of \begin{equation*}
\dot{u}(t)+A\int_0^tb(t-s)u(s)\,\dd s=0,\ t>0;\quad u(0)=u_0.  
\end{equation*}
We refer to the monograph \cite{PrussBook1993} for a comprehensive
theory of resolvent families for Volterra equations.  An important
feature of the resolvent family $\{S(t)\}_{t\ge 0}$ is that it does
not have the semigroup property; that is, $S(t+s)\neq S(t)S(s)$. This
is the mathematical reflection of the fact that the solution possesses
a nontrivial memory.  In our special setting, using the spectral
decomposition of $A$, an explicit representation of $S(t)$ is given by
the Fourier series
\begin{equation} \label{SDecomposition}
  S(t)v=\sum_{k=1}^\infty s_k(t) (v,\vp_k)\vp_k,
\end{equation}
where the functions $s_k(t)$ are the solutions of 
\begin{equation} \label{sk}
  \dot s_k(t)+\lambda_k \int_0^t b(t-s)s_k(s) \, \dd s=0,
  \ t>0;\quad s_k(0)=1.
\end{equation}

Next, we collect our precise assumptions on the resolvent family
$\{S(t)\}_{t\geq 0}$.

\begin{assumption} \label{Assumption_S} We assume that the resolvent
  family $\{S(t)\}_{t\ge0}$ is strongly continuously differentiable
  for $t>0$ and enjoys the following smoothing properties: There is
  $M$ such that for $t>0$, we have
\begin{align}  
  &\|A^s S(t)\|_{\cL  } \leq M t^{-s\rho},
    \quad &&s\in\Big[0,\frac{1}{\rho}\Big],
     \label{S:1}\\
  &\|A^s\dot S(t)\|_{\cL  } 
    \leq M t^{-s\rho-1},
     \quad &&s\in\Big[0,\frac{1}{\rho}\Big],
      \label{S:2}\\
  &\|A^{-s}\dot S(t)\|_{\cL  } 
    \leq M t^{s\rho-1},
     \quad &&s\in[0,1].
      \label{S:3}
\end{align}
\end{assumption}

\begin{rem}
  \label{rem:S4}
These are verified in 
\cite[Theorem 5.5]{McLeanThomee} for the Riesz kernel and in 
\cite[Lemma A.4]{BaeumerGeissertKovacs} for more general 
kernels.  We note that for the Riesz kernel 
\eqref{RieszKernel}, which is our main example, 
estimates \eqref{S:1} and \eqref{S:2} hold also for 
$s\in [0,1]$, see \cite[Theorem 5.5]{McLeanThomee}, 
but we do not need this extended range of $s$ for the 
present analysis.  A more general class of kernels $b$ 
for which \eqref{S:1}--\eqref{S:3} are satisfied
is the the class of $4$-monotone kernels such that 
$0\neq b\in L_{1,\mathrm{loc}}(\mathbb{R}_+)$,  
$\lim_{t\to \infty}b(t)=0$, with 
\begin{equation*}
  \rho  := 1 + \frac{2}{\pi}
  \sup \{ | \mathrm{arg} \, \widehat{b}(z) | : \mathrm{Re}\,z >0 \} \in (1,2), 
\end{equation*} 
and $\widehat{b}(z)\le C z^{1-\rho}$ for $z>1$, where
this latter condition may be substituted by the condition
$\|b\|_{L_1(0,t)}\le Ct^{\rho-1}$, $t\in (0,1)$, see \cite[Remark 3.8
and Lemma A.4]{BaeumerGeissertKovacs}.  In particular, $b$ does not
have to be analytic.  (Here $\widehat{b}$ denotes the Laplace transform
of $b$.)
\end{rem}

\subsection{Well-posedness of the semilinear stochastic problem}
The mild solution of the semilinear stochastic equation
\eqref{VolterraEq} is an adapted $H$-valued stochastic process,
$u(t)$, such that, for $t\in[0,T]$, $\bP$-almost surely,
\begin{equation} \label{VarConstForm} 
  u(t)=S(t)u_0 
   + \int_0^t  S(t-s)F(u(s))\, \dd s 
   +\int_0^t S(t-s)\, \dd W(s).
\end{equation}

\begin{assumption} \label{Assumption_F}
In addition to the singularity exponent 
$\rho=\alpha+1\in(1,2)$ from
\eqref{RieszKernel} and the regularity parameter
$\beta\in (0,\frac{1}{\rho}]$ in \eqref{NoiseRegularity}, 
we assume that there are $\delta\in [1,\frac{2}{\rho})$, 
$\gamma\in [0,\beta)$, $\eta \in [1,\frac{2}{\rho})$, 
and a constant $L>0$, such that
\begin{align}  
  &\|F(u)\|\leq L(1+\|u\|),\quad 
  \|F'(u)v\|\leq L\|v\|,  &&u,v\in H, 
   \label{F:1}\\
  &\|F'(u)v\|_{-\delta} 
    \leq L(1+\|u\|_{\gamma})\|v\|_{-\gamma},
     &&u\in \dot H^{\gamma},
     \ v\in \dot H^{-\gamma}, 
    \label{F:2}\\
  &\|F''(u)(v_1,v_2)\|_{-\eta} 
    \leq L\|v_1\|\|v_2\|,  &&v_1,v_2\in H.  
   \label{F:3}
\end{align}
\end{assumption}

Our main example is $H=L_2(\cD)$ with $\cD\subset \IR^d$ a bounded
domain with appropriately smooth boundary, and $A=-\Delta $, the
negative of the Dirichlet Laplacian.  Here $F$ can be taken to be a
Nemytskij operator defined by $F(u)(x)=f(u(x))$, where
$f\colon\IR\to\IR$ is a smooth function with bounded derivatives of
orders 1 and 2.  Then \eqref{F:1} clearly holds and \eqref{F:3} is
satisfied with $\eta>d/2$ because of Sobolev's inequality.  The
additional assumption $\eta <\frac{2}{\rho}$ puts a restriction on
$\rho$, namely, $1<\rho<4/d$.  For \eqref{F:2} we refer to Lemma~4.4
in \cite{WangGanTang}, which can be extended from $d=1$ to $d\le3$,
again in case $\delta>d/2$ and thus $1<\rho<4/d$.

\begin{lemma} \label{SolutionRegularityLemma} 
Suppose that Assumption \ref{Assumption_W}, \eqref{S:1} 
from Assumption \ref{Assumption_S}, and \eqref{F:1} 
from Assumption \ref{Assumption_F} hold. 
Let $p\geq2$, and assume 
$\|u_0\|_{L_p(\Omega;\dot H^\gamma)}\le K$. 
Then, there is a unique mild solution 
$u\in \cC([0,T]; L_p(\Omega;H))$ of \eqref{VarConstForm}. 
Furthermore, for a constant 
$C=C(\QQ,K,L,M,T,\beta,\gamma,\rho,p)$, 
\begin{equation} \label{SolutionRegularity}
 \sup_{t\in [0,T]}\|u(t)\|_{L_p(\Omega;
  \dot H^\gamma)}
 \leq C . 
\end{equation}
\end{lemma}

\begin{proof}
The existence and uniqueness of a mild solution
$u\in \cC([0,T]; L_p(\Omega;H))$ of \eqref{VarConstForm} 
can be proved, even only under assumption~\eqref{F:1}, 
via a standard Banach fixed point argument using 
\eqref{NoiseRegularity} and \eqref{S:1}, 
see, for example, the proof of 
\cite[Theorem 3.3]{BaeumerGeissertKovacs}.
Therefore,
\begin{equation} \label{SolutionRegularityGamma0}
  \|u(t)\|_{L_p(\Omega;H)}
   \leq C,    \quad t\in [0,T],
\end{equation}
which is \eqref{SolutionRegularity} with 
$\gamma=0$. 
For $\gamma\in (0,\beta)$, using 
\eqref{VarConstForm}, we have
\begin{equation*} 
 \begin{split}
  \|u(t)\|_{L_p(\Omega;\dot H^\gamma)}
   &\leq \|S(t)\|_{\cL  }\|u_0\|_{L_p(\Omega;
     \dot H^\gamma)} \\
   &\quad +\int_0^t \|A^{\frac{\gamma}{2}}
     S(t-s)\|_{\cL  } 
      \|F(u(s))\|_{L_p(\Omega;H)}\, \dd s \\
   &\quad +\Big\|\int_0^t A^{\frac{\gamma}{2}}S(t-s)\, 
        \dd W(s)\Big\|_{L_p(\Omega;H)}. 
 \end{split}
\end{equation*}
By using \eqref{S:1} with $s=0$, \eqref{S:1}, \eqref{F:1},  
\eqref{BDG}, and \eqref{SolutionRegularityGamma0},  
we obtain
\begin{equation*} 
 \begin{split}
  \|u(t)\|_{L_p(\Omega;\dot H^\gamma)}
   &\leq \|u_0\|_{L_p(\Omega;\dot H^\gamma)} 
    +L\int_0^t \|A^{\frac{\gamma}{2}}S(t-s)\|_{\cL}
     \big(1+\|u(s)\|_{L_p(\Omega;H)}\big)
      \, \dd s \\
   &\quad +C_p\Big\|\Big(\int_0^t 
     \|A^{\frac{\gamma}{2}}S(t-s)
      Q^{\frac{1}{2}}\|_{\HS}^2\, \dd s
       \Big)^{\frac{1}{2}}\Big\|_{L_p(\Omega;\IR)}\\
   &\leq C+C \int_0^t (t-s)^{-\frac{\gamma\rho}{2}}\, 
     \dd s \\
   &\quad +C_p
  \big\|A^{\frac{\beta-\frac{1}{\rho}}{2}}  
   Q^{\frac{1}{2}}\big\|_{\HS} \Big(\int_0^t 
    \Big\| A^{\frac{(\gamma-\beta) 
     +\frac{1}{\rho}}{2}} S(t-s)\Big\|^2\, \dd s
    \Big)^{\frac{1}{2}}.  
 \end{split}
\end{equation*}
By using \eqref{NoiseRegularity} and \eqref{S:1} again, 
we have 
\begin{equation*} 
 \begin{split}
  \|u(t)\|_{L_p(\Omega;\dot H^\gamma)}
  \leq C +BC
   \Big(\int_0^t 
      (t-s)^{-1+(\beta-\gamma)\rho}\, \dd s
       \Big)^{\frac{1}{2}}, 
 \end{split}
\end{equation*}
where the integral is finite, since 
$(\beta-\gamma)\rho\in (0,1)$.
This completes the proof.
\end{proof}

\begin{rem} \label{DeterministicSolutionRegularity}
In the deterministic case, i.e., when $W=0$, 
by following the proof of 
Lemma~\ref{SolutionRegularityLemma}, it is straightforward 
to prove that, assuming 
$u_0\in \dot H^{2\gamma}$ for some 
$\gamma\in [0,\frac{1}{\rho})$,  
we have the regularity estimate 
\begin{equation} 
 \label{DeterministicSolutionRegularityEst}
   \sup_{t\in [0,T]}\|u(t)\|_{\dot H^{2\gamma}}
    \leq C.
\end{equation}
\end{rem}

  \section{Full discretization} \label{Sec3}
  In this section we formulate a fully discrete method for
  approximation of \eqref{VolterraEq}. We use the spectral Galerkin
  method for spatial discretization in combination with a time
  discretization based on an exponential Euler type method.  We refer
  to the proposed time discretization method as the Mittag-Leffler
  Euler integrator (MLEI), since the solution operator can be
  represented using the Mittag-Leffler function, in case the
  convolution kernel $b$ is the Riesz kernel as in our main example
  \eqref{Volterrafrac}.  We give more details in
  Section~\ref{Sec:Numeric}, where numerical examples are presented.

Let $0=t_0<t_1< \dots <t_M=T$ be a uniform partition of 
the time interval $[0,T]$, with time step 
$\Delta t=t_{m+1}-t_m,\ m=0,1,\dots,M-1$. 
Then, for $m=0,1,\dots,M$, 
by using the variation of constants formula 
\eqref{VarConstForm}, we have
\begin{equation} \label{um}
 \begin{split}
  u(t_m)&=S(t_m)u_0+\int_0^{t_m}
   S(t_m-\sigma)F(u(\sigma))
    \, \dd\sigma+\int_0^{t_m} S(t_m-\sigma)\, 
     \dd W(\sigma)\\
  &=S(t_m)u_0 + \sum_{j=0}^{m-1}
   \int_{t_{j}}^{t_{j+1}}S(t_m-\sigma)F(u(\sigma))
    \, \dd\sigma+\int_0^{t_m}S(t_m-\sigma)
     \, \dd W(\sigma),
 \end{split}
\end{equation}
Following the idea of exponential integrators, we 
formulate the MLEI as
\begin{equation} \label{Um}
 \begin{split}
  U_m&=S(t_m)u_0 + \sum_{j=0}^{m-1}
  \int_{t_{j}}^{t_{j+1}} S(t_m-\sigma)\, \dd \sigma 
   \, F(U_j) +\int_0^{t_m} S(t_m-\sigma)
    \, \dd W(\sigma),
 \end{split}
\end{equation}
where $U_j\approx u(t_j)$, $j=0,1,...,M$, and where the convolution containing 
the nonlinear term is approximated but the linear terms, 
including the stochastic convolution integral, 
are computed exactly, see 
Section \ref{Sec:Numeric} for details.

For spatial discretization, we define finite-dimensional 
subspaces $H_N$ of $H$ by 
$H_N=\Span\{\vp_1,\vp_2,\dots,\vp_N\}$, where 
$\{\vp_k\}_{k=1}^{\infty}$ are the eigenvectors of $A$,
\eqref{EigenEq}.  Then we define the projector
\begin{equation} \label{PN}
 \cP_N\colon H\to H_N,\quad 
  \cP_N v=\sum_{k=1}^N ( v,\vp_k)
   \vp_k,\quad v\in H. 
\end{equation}
We also define the operator 
\begin{equation} \label{AN}
 A_N\colon H_N \to H_N,\quad A_N=A \cP_N, 
\end{equation}
which generates a family of resolvent operators 
$\{S_N(t)\}_{t\geq 0}$ in $H_N$. 
It is known that 
\begin{align}
 &S_N(t)\cP_N=S(t)\cP_N,\label{SNPN}\\
 &\|A^{-\nu}(I-\cP_N)\|
  =\sup_{k\geq N+1}\lambda_k^{-\nu}
  =\lambda_{N+1}^{-\nu},\quad \nu \geq 0.  
  \label{PN:est}
\end{align}

The representation of $S_N$, similar to 
\eqref{SDecomposition}, is given by 
\begin{equation*} 
  S_N(t)v=\sum_{k=1}^N s_{k}(t) (v,\vp_k)\vp_k. 
\end{equation*}
Therefore, the smoothing properties  
\eqref{S:1}--\eqref{S:3} also hold for $S_N$ with constants
  independent of $N$.  

Hence, the fully discrete approximation of 
\eqref{VolterraEq}, based on the temporal approximation
\eqref{Um}, is given by 
\begin{equation} \label{UmN}
 \begin{split}
   U_m^N=S_N(t_m)\cP_N u_0 
    &+ \sum_{j=0}^{m-1}
     \int_{t_{j}}^{t_{j+1}} S_N(t_m-\sigma)
      \, \dd \sigma 
       \, \cP_N F(U_j^N) \\
    &+\int_0^{t_m} S_N(t_m-\sigma) \cP_N
       \, \dd W(\sigma),
 \end{split}
\end{equation}
with initial value $U_0^N = \cP_N u_0$.  
Now we state and prove the main theorem, that shows 
the strong rate of convergence. 

\begin{theorem} \label{thm:main} Suppose that
  Assumption~\ref{Assumption_W}, Assumption~\ref{Assumption_S}, and
  Assumption~\ref{Assumption_F} hold and
  $\|u_0\|_{L_4(\Omega;\dot H^\gamma)}\le K$.  Then, for a constant
  $C=C(\QQ,K,L,T,\beta,\rho,\gamma)$, we have
\begin{equation*} 
\sup_{t_m\in[0,T]} \|u(t_m)-U_m^N\|_{L_2(\Omega; H)}
 \leq C \big(\lambda_N^{-\frac{\gamma}{2}}
      +\Delta t^{\gamma\rho}\big).
\end{equation*}
\end{theorem}

\begin{proof}
By subtracting \eqref{UmN} from \eqref{um}, 
we get
\begin{equation*} 
 \begin{split}
  u(t_m)-U_m^N&=S(t_m)u_0-S_N(t_m)\cP_N u_0\\
   &\quad +\sum_{j=0}^{m-1}
      \int_{t_{j}}^{t_{j+1}}
      \Big\{S(t_m-\sigma)F(u(\sigma))
    -S_N(t_m-\sigma)\cP_N F(U_j^N)\Big\}\,\dd\sigma\\
   &\quad +\int_{0}^{t_m} 
    \big\{S(t_m-\sigma)-S_N(t_m-\sigma)\cP_N\big\} 
     \  \dd W(\sigma).
 \end{split}
\end{equation*}
By recalling \eqref{SNPN} and taking norms, we obtain
\begin{equation} \label{I:s}
 \begin{split}
  &\|u(t_m)-U_m^N\|_{L_2(\Omega; H)}
\leq\|S(t_m)(I-\cP_N) u_0\|_{L_2(\Omega; H)}\\
   &\qquad\quad +\Big\|
      \int_{0}^{t_{m}}
       S(t_m-\sigma)(I-\cP_N)F(u(\sigma))
        \,\dd\sigma\Big\|_{L_2(\Omega; H)}\\
   &\qquad\quad +\Big\|\sum_{j=0}^{m-1}
      \int_{t_{j}}^{t_{j+1}}
       S(t_m-\sigma)\cP_N \big(
        F(u(\sigma))-F(U_j^N)\big)\,\dd\sigma
         \Big\|_{L_2(\Omega; H)}\\
   &\qquad\quad +\Big\|\int_{0}^{t_m} 
     S(t_m-\sigma)(I-\cP_N)\,\dd W(\sigma)
      \Big\|_{L_2(\Omega; H)}
 =\sum_{l=1}^4 I_l.
 \end{split}
\end{equation}
We note that $I_1,I_2$, and $I_4$ correspond to 
the \emph{spatial discretization error}, while $I_3$ 
corresponds to the \emph{temporal error}. 

1. {\bf Spatial error}: The estimate of $I_1$ is a  
consequence of \eqref{S:1} with $s=0$ and \eqref{PN:est}, as
\begin{equation} \label{I1}
  \begin{split}
  I_1
  &\leq \|S(t_m)\|_{\cL  }
   \|A^{-\frac{\gamma}{2}}(I-\cP_N) 
    A^{\frac{\gamma}{2}}u_0\|_{L_2(\Omega; H)}
\\ &  \leq C \lambda_{N+1}^{-\frac{\gamma}{2}}
    \|u_0\|_{L_2(\Omega;\dot H^\gamma)}
\leq C \lambda_{N+1}^{-\frac{\gamma}{2}}   .
  \end{split}
\end{equation}

For $I_2$, by using \eqref{S:1} 
and \eqref{PN:est}, 
we have
\begin{equation} \label{I2}
 \begin{split}
  I_2
 &\leq \int_{0}^{t_{m}}
   \|A^{\gamma}S(t_m-\sigma)\|_{\cL  }\\
 &\qquad \times\|A^{-\gamma}(I-\cP_N)
   \|_{\cL  }
    \|F(u(\sigma))\|_{L_2(\Omega; H)}\,\dd \sigma\\
 &\leq C \int_{0}^{t_{m}}
   (t_m-\sigma)^{-\gamma\rho}
    \lambda_{N+1}^{-\gamma}
    \|F(u(\sigma))\|_{L_2(\Omega; H)}\,\dd \sigma
   \leq C \lambda_{N+1}^{-\gamma}, 
 \end{split}
\end{equation}
where we recall that $\gamma\rho<1$ and use \eqref{F:1} and
\eqref{SolutionRegularity} with $p=2, \gamma=0$.

Now we estimate $I_4$. 
Using the It\^o isometry \eqref{ItoIsometry}, 
we have
\begin{equation*} 
 \begin{split}
  I_4
  &\leq  \Big\|\Big(\int_{0}^{t_m} 
     \big\|S(t_m-\sigma)(I-\cP_N)Q^{\frac{1}{2}}
      \big\|_{\HS}^2
      \,\dd \sigma \Big)^{\frac{1}{2}}
      \Big\| \\ 
  &\leq \big\|A^{\frac{\beta-\frac{1}{\rho}}{2}} 
    Q^{\frac{1}{2}}\big\|_{\HS}
     \|A^{-\frac{\gamma}{2}}(I-\cP_N)\|_{\cL  }
 \Big(\int_{0}^{t_m} 
     \big\|A^{\frac{(\gamma-\beta)
      +\frac{1}{\rho}}{2}}
       S(t_m-\sigma)\big\|_{\HS}^2
       \,\dd \sigma \Big)^{\frac{1}{2}},
 \end{split}
\end{equation*}
which, by \eqref{S:1}, \eqref{PN:est}, and since
$(\beta-\gamma)\rho\in (0,1)$, implies
\begin{align} \label{I4}
  \begin{split}
I_4
 &\leq C \lambda_{N+1}^{-\frac{\gamma}{2}} 
  \big\|A^{\frac{\beta-\frac{1}{\rho}}{2}}
  Q^{\frac{1}{2}}\big\|_{\HS}
   \Big(\int_{0}^{t_m} 
    (t_m-\sigma)^{-1+(\beta-\gamma)\rho}
      \,\dd \sigma \Big)^{\frac{1}{2}}
\\ & 
\leq C \lambda_{N+1}^{-\frac{\gamma}{2}} .  
  \end{split}
\end{align}

2. {\bf Temporal error}: 
Here we estimate $I_3$, i.e.,
\begin{equation*}
  I_3
   = \Big\|\sum_{j=0}^{m-1}
      \int_{t_{j}}^{t_{j+1}}
       S(t_m-\sigma)\cP_N \big(
        F(u(\sigma))-F(U_j^N)\big)\,\dd\sigma
         \Big\|_{L_2(\Omega; H)}.
\end{equation*}
We use the Taylor expansion 
\begin{equation*} 
  F(u(\sigma))=F(u(t_j))
   +F'(u(t_j))\big(u(\sigma)-u(t_j)\big)
    +R_{F,j}(\sigma)
\end{equation*}
where the remainder is
\begin{equation*} 
  R_{F,j}(\sigma)=\int_0^1 F''\Big(u(t_j)
     +\gamma \big(u(\sigma)-u(t_j)\big)\Big)
      \big(u(\sigma)-u(t_j),u(\sigma)-u(t_j)\big)
       (1-\gamma)\,\dd \gamma,
\end{equation*}
to get
\begin{equation*} 
 \begin{split}
  I_3
    &\leq \Big\|\sum_{j=0}^{m-1}
      \int_{t_{j}}^{t_{j+1}}S(t_m-\sigma)\cP_N 
       \big(F(u(t_j))-F(U_j^N)\big)
        \,\dd\sigma\Big\|_{L_2(\Omega; H)}\\
    &\quad +\Big\|\sum_{j=0}^{m-1}
      \int_{t_{j}}^{t_{j+1}}S(t_m-\sigma)\cP_N 
       F'(u(t_j))\big(u(\sigma)-u(t_j)\big)
        \,\dd\sigma\Big\|_{L_2(\Omega; H)}\\
    &\quad +\Big\|\sum_{j=0}^{m-1}
      \int_{t_{j}}^{t_{j+1}}S(t_m-\sigma)
       \cP_N R_{F,j}(\sigma) 
        \,\dd\sigma\Big\|_{L_2(\Omega; H)}.   
 \end{split}
\end{equation*}
By substituting $u(\sigma)$ and $u(t_j)$ from the variation 
of constants formula \eqref{VarConstForm} in the second term, 
we have
\begin{equation} \label{I3:s}
 I_3 \le \sum_{l=1}^7 I_{3,l}, 
\end{equation}
where
\begin{align*}
 I_{3,1}= 
\Big\|\sum_{j=0}^{m-1}
\int_{t_{j}}^{t_{j+1}}S(t_m-\sigma)\cP_N 
\big(F(u(t_j))-F(U_j^N)\big) 
\,\dd\sigma\Big\|_{L_2(\Omega; H)}  , 
\end{align*}
\begin{align*}
 I_{3,2}= 
\Big\|\sum_{j=0}^{m-1}
\int_{t_{j}}^{t_{j+1}}S(t_m-\sigma)\cP_N 
F'(u(t_j))\big(S(\sigma)-S(t_j)\big)u_0
\, \dd\sigma\Big\|_{L_2(\Omega; H)}, 
\end{align*}
\begin{align*}
 I_{3,3}= 
\Big\|\sum_{j=0}^{m-1}
\int_{t_{j}}^{t_{j+1}}S(t_m-\sigma)\cP_N  F'(u(t_j))
\int_{t_j}^{\sigma} 
S(\sigma-\tau) F(u(\tau)) \,\dd\tau
\,\dd\sigma\Big\|_{L_2(\Omega; H)}, 
\end{align*}
\begin{align*}
 I_{3,4}&= 
\Big\|\sum_{j=0}^{m-1}
      \int_{t_{j}}^{t_{j+1}}S(t_m-\sigma)\cP_N 
       F'(u(t_j))
\\ & \quad 
\times\int_0^{t_j} 
        \big(S(\sigma-\tau)-S(t_j-\tau)\big) 
         F(u(\tau)) \,\dd\tau
          \,\dd\sigma\Big\|_{L_2(\Omega; H)},
\end{align*}
\begin{align*}
 I_{3,5}= 
\Big\|\sum_{j=0}^{m-1}
      \int_{t_{j}}^{t_{j+1}}S(t_m-\sigma)\cP_N 
       F'(u(t_j)) 
\int_{t_j}^{\sigma} 
        S(\sigma-\tau) \,\dd W(\tau)
        \,\dd\sigma\Big\|_{L_2(\Omega; H)}, 
\end{align*}
\begin{align*}
 I_{3,6}&= 
\Big\|\sum_{j=0}^{m-1}
      \int_{t_{j}}^{t_{j+1}}S(t_m-\sigma)\cP_N 
       F'(u(t_j))
\\ & \quad \times \int_0^{t_j} 
        \big(S(\sigma-\tau)-S(t_j-\tau)\big) 
         \,\dd W(\tau)
          \,\dd\sigma\Big\|_{L_2(\Omega; H)}, 
\end{align*}
and 
\begin{align*}
 I_{3,7}= \Big\|\sum_{j=0}^{m-1}
  \int_{t_{j}}^{t_{j+1}}S(t_m-\sigma)
   \cP_N R_{F,j}(\sigma)\,\dd\sigma\Big\|_{L_2(\Omega; H)}. 
\end{align*}

First, using \eqref{F:1} and 
\eqref{S:1} with $s=0$, we have
\begin{equation} \label{I31}
  I_{3,1}
  \leq ML \Delta t\sum_{j=0}^{m-1} 
    \|u(t_j)-U_j^N\|_{L_2(\Omega;H)}.
\end{equation}
To estimate $I_{3,2}$, we have 
\begin{equation*} 
 \begin{split}
  I_{3,2}
 &\leq \sum_{j=0}^{m-1}\int_{t_j}^{t_{j+1}} 
   \|A^{\frac{\delta}{2}} S(t_m-\sigma)\|_{\cL  }
    \|A^{-\frac{\delta}{2}}F'(u(t_j))
     \big(S(\sigma)-S(t_j)\big)u_0\|_{L_2(\Omega;H)}
      \,\dd\sigma\\
 &= \sum_{j=0}^{m-1}\int_{t_j}^{t_{j+1}} 
   \|A^{\frac{\delta}{2}} S(t_m-\sigma)\|_{\cL  }
    \Big\|A^{-\frac{\delta}{2}}F'(u(t_j))
     \int_{t_j}^\sigma \dot S(\tau)u_0\,\dd \tau
      \Big\|_{L_2(\Omega;H)}\,\dd\sigma, 
 \end{split}
\end{equation*}
so that, using \eqref{F:2}, \eqref{S:1}, and 
\eqref{SolutionRegularity} with $p=4$, we obtain
\begin{equation*} 
 \begin{split}
  I_{3,2} 
 &\leq C 
\sum_{j=0}^{m-1}\int_{t_j}^{t_{j+1}}
  (t_m-\sigma)^{-\frac{\delta\rho}{2}}\big(1+\|u(t_j)
  \|_{L_4(\Omega;\dot{H}^\gamma)}\big)
      \\ & \quad \times
     \Big\|\Big\|\int_{t_j}^\sigma \dot S(\tau)u_0
      \,\dd \tau\Big\|_{-\gamma}
       \Big\|_{L_{4}(\Omega;\IR)}\,\dd\sigma \\
 &\leq C
 \sum_{j=0}^{m-1}\int_{t_j}^{t_{j+1}}
        (t_m-\sigma)^{-\frac{\delta\rho}{2}}
    \Big\|\int_{t_j}^\sigma \big\|A^{-\gamma}
     \dot S(\tau)A^{\frac{\gamma}{2}}u_0
      \,\dd \tau\big\|
       \Big\|_{L_{4}(\Omega;\IR)}\,\dd\sigma.
 \end{split}
\end{equation*}
Now, by \eqref{S:3}, we have
\begin{equation*} 
 \begin{split}
  I_{3,2}
  &\leq C\|u_0\|_{L_{4}(\Omega;\dot H^{\gamma})}
      \sum_{j=0}^{m-1}\int_{t_j}^{t_{j+1}}
         (t_m-\sigma)^{-\frac{\delta\rho}{2}}
\int_{t_j}^\sigma \tau^{\gamma\rho-1}
       \,\dd \tau \,\dd\sigma \\
  &\leq \frac{C}{\gamma\rho}
       (t_{j+1}^{\gamma\rho}-t_j^{\gamma\rho}) 
        \sum_{j=0}^{m-1}\int_{t_j}^{t_{j+1}}
         (t_m-\sigma)^{-\frac{\delta\rho}{2}} 
          \,\dd\sigma
\leq C 
       (t_{j+1}^{\gamma\rho}-t_j^{\gamma\rho}),
 \end{split}
\end{equation*} 
and, since $\gamma\rho\in (0,1)$, we 
consequently have
\begin{equation} \label{I32}
 \begin{split}
  I_{3,2}  \leq C  \Delta t^{\gamma\rho}.
 \end{split}
\end{equation}

Now we estimate $I_{3,3}$ in \eqref{I3:s}. 
Using \eqref{F:1} and \eqref{S:1} with $s=0$, we have
\begin{equation*} 
 \begin{split}
  I_{3,3}
 &\leq \sum_{j=0}^{m-1}\int_{t_j}^{t_{j+1}} 
   \|S(t_m-\sigma)\|_{\cL  }
    \Big\|F'(u(t_j))
     \int_{t_j}^{\sigma} 
       S(\sigma-\tau) F(u(\tau)) \,\dd\tau\ 
      \Big\|_{L_2(\Omega;H)}
      \,\dd\sigma\\
 &\leq L\sum_{j=0}^{m-1}\int_{t_j}^{t_{j+1}} 
   \|S(t_m-\sigma)\|_{\cL  }
    \int_{t_j}^{\sigma} 
     \|S(\sigma-\tau)\|_{\cL  } 
      \|F(u(\tau))\big\|_{L_2(\Omega;H)} \,\dd\tau
       \,\dd\sigma \\
 &\leq C\sum_{j=0}^{m-1}\int_{t_j}^{t_{j+1}} 
   \int_{t_j}^{\sigma} \big(1+\|u(\tau)
    \|_{L_2(\Omega;H)}\big)\,\dd\tau \,\dd\sigma,
 \end{split}
\end{equation*}
that, by \eqref{SolutionRegularity} with 
$p=2, \gamma=0$, implies 
\begin{equation} \label{I33}
  I_{3,3}  \leq  C\Delta t.
\end{equation}

To estimate $I_{3,4}$ in \eqref{I3:s}, we have
\begin{equation*}
 \begin{split}
  I_{3,4}
 &\leq \sum_{j=0}^{m-1}\int_{t_j}^{t_{j+1}} 
   \|A^{\frac{\delta}{2}} S(t_m-\sigma)\|_{\cL  } \\
 &\quad\times \Big\|A^{-\frac{\delta}{2}}F'(u(t_j))
     \int_0^{t_j}\big(S(\sigma-\tau)-S(t_j-\tau)\big) 
      F(u(\tau)) \,\dd\tau\ \Big\|_{L_2(\Omega;H)}
       \,\dd\sigma\\
 &= \sum_{j=0}^{m-1}\int_{t_j}^{t_{j+1}} 
   \|A^{\frac{\delta}{2}} S(t_m-\sigma)\|_{\cL  } \\
 &\quad\times\Big\|\Big\|A^{-\frac{\delta}{2}}
   F'(u(t_j))
    \int_0^{t_j} \int_{t_j}^\sigma 
     \dot S(\theta-\tau)\,\dd \theta \ 
      F(u(\tau))\,\dd\tau\ \Big\|
       \Big\|_{L_2(\Omega;\IR)}
        \,\dd\sigma,  
 \end{split}
\end{equation*}
which, in view of \eqref{S:1}, \eqref{F:2}, and
\eqref{SolutionRegularity} with $p=2$, implies
\begin{equation*} 
 \begin{split}
  I_{3,4}
&\leq C 
      \sum_{j=0}^{m-1}\int_{t_j}^{t_{j+1}} 
       (t_m-\sigma)^{-\frac{\delta\rho}{2}}\\
  &\quad\times\Big\|\Big\|
      \int_0^{t_j} \int_{t_j}^\sigma 
       A^{-\frac{\gamma}{2}}
       \dot S(\theta-\tau)\,\dd \theta \ 
        F(u(\tau))\,\dd\tau\, \Big\|
        \Big\|_{L_{4}(\Omega;\IR)}
         \,\dd\sigma,  \\
  &\leq C 
      \sum_{j=0}^{m-1}\int_{t_j}^{t_{j+1}} 
       (t_m-\sigma)^{-\frac{\delta\rho}{2}}  \\
  &\quad\times \int_0^{t_j} \int_{t_j}^\sigma 
      \big\|A^{-\frac{\gamma}{2}}
       \dot S(\theta-\tau)\big\|_{\cL  }\,\dd\theta\ 
        \big\|F(u(\tau))\big\|_{L_{4}(\Omega;H)}
         \,\dd\tau\,\dd\sigma.
 \end{split}
\end{equation*}
Now, by \eqref{S:3} and \eqref{F:1}, we have 
\begin{equation*} 
 \begin{split}
  I_{3,4}
  &\leq C 
      \sum_{j=0}^{m-1}\int_{t_j}^{t_{j+1}} 
       (t_m-\sigma)^{-\frac{\delta\rho}{2}}  \\
  &\quad\times \int_0^{t_j} \int_{t_j}^\sigma 
     (\theta-\tau)^{\frac{\gamma\rho}{2}-1}
      \,\dd\theta\ \big(1+\|u(\tau)
       \|_{L_{4}(\Omega;H)}\big)
         \,\dd\tau\,\dd\sigma , 
 \end{split}
\end{equation*}
which, together with \eqref{SolutionRegularity} 
with $p=4, \gamma=0$,  
implies
\begin{equation*} 
 I_{3,4}
  \leq C 
     \sum_{j=0}^{m-1}\int_{t_j}^{t_{j+1}} 
      (t_m-\sigma)^{-\frac{\delta\rho}{2}}  
   \int_0^{t_j} \int_{t_j}^\sigma 
     (\theta-\tau)^{\frac{\gamma\rho}{2}-1}
      \,\dd\theta\ \,\dd\tau\,\dd\sigma . 
\end{equation*}
Then, computing the double integral as 
\begin{equation*} 
 \begin{split}
  \int_0^{t_j} \int_{t_j}^\sigma 
   (\theta-\tau)^{\frac{\gamma\rho}{2}-1}
    \,\dd\theta\,\dd\tau
  &=\int_{t_j}^\sigma \int_0^{t_j}
    (\theta-\tau)^{\frac{\gamma\rho}{2}-1}
     \,\dd\tau\,\dd\theta\\
  &=\frac{2}{\gamma\rho}
     \int_{t_j}^\sigma \big(
      \theta^{\frac{\gamma\rho}{2}}
       -(\theta-t_j)^{\frac{\gamma\rho}{2}}\big) 
        \,\dd\theta 
  \leq \frac{2}{\gamma\rho} 
    t_j^{\frac{\gamma\rho}{2}} \Delta t,
 \end{split}
\end{equation*}
due to $\frac{\gamma\rho}{2}\in (0,\frac{1}{2})$,
we conclude the estimate
\begin{equation} \label{I34}
  I_{3,4}
  \leq C \Delta t.
\end{equation}

We now estimate the terms in \eqref{I3:s}, which are 
affected by the noise. 
For $I_{3,5}$, using the fact that the expected value 
of independent processes is zero, and then the
Cauchy--Schwarz inequality, we have
\begin{equation*} 
 \begin{split}
  I_{3,5}^2
   &=\mathbb{E}\Big\|\sum_{j=0}^{m-1}
      \int_{t_{j}}^{t_{j+1}}S(t_m-\sigma)\cP_N 
       F'(u(t_j))\int_{t_j}^{\sigma} 
        S(\sigma-\tau) \,\dd W(\tau)
        \,\dd\sigma\Big\|^2\\
   &=\sum_{j=0}^{m-1} \mathbb{E}\Big\|
      \int_{t_{j}}^{t_{j+1}}\int_{t_j}^{\sigma}
       S(t_m-\sigma)\cP_N F'(u(t_j)) 
        S(\sigma-\tau) \,\dd W(\tau)
        \,\dd\sigma\Big\|^2\\
   &\leq \Delta t \sum_{j=0}^{m-1}
     \int_{t_{j}}^{t_{j+1}}\mathbb{E}\Big\|
      \int_{t_j}^{\sigma}
       S(t_m-\sigma)\cP_N F'(u(t_j)) 
        S(\sigma-\tau) \,\dd W(\tau)
        \Big\|^2\,\dd\sigma.
 \end{split}
\end{equation*}
Then, by the It\^o isometry \eqref{ItoIsometry}, 
\eqref{F:1} and \eqref{SolutionRegularity} 
with $p=2, \gamma=0$, we have
\begin{equation*} 
 \begin{split}
  I_{3,5}^2&\leq \Delta t \sum_{j=0}^{m-1}
     \int_{t_{j}}^{t_{j+1}}\int_{t_j}^{\sigma}
      \Big\| S(t_m-\sigma)\cP_N F'(u(t_j)) 
        S(\sigma-\tau) Q^{\frac{1}{2}}
        \Big\|_{\HS}^2 \,\dd \tau\,\dd\sigma  \\
   &\leq C \Delta t
     \big\|A^{\frac{\beta-\frac{1}{\rho}}{2}}
       Q^{\frac{1}{2}}\big\|_{\HS}^2
     \sum_{j=0}^{m-1}
      \int_{t_{j}}^{t_{j+1}}\int_{t_j}^{\sigma}
       \Big\| S(t_m-\sigma)\Big\|_{\cL  }^2
        \Big\|A^{\frac{-\beta+\frac{1}{\rho}}{2}} 
         S(\sigma-\tau) \Big\|_{\cL  }^2 
         \,\dd \tau\,\dd\sigma.
 \end{split}
\end{equation*}
Now, using \eqref{NoiseRegularity} and
\eqref{S:1}, we obtain
\begin{equation*} 
  I_{3,5}^2
  \leq C\Delta t
     \sum_{j=0}^{m-1}
      \int_{t_{j}}^{t_{j+1}}\int_{t_j}^{\sigma}
       (\sigma-\tau)^{\beta\rho-1} 
        \,\dd \tau\,\dd\sigma 
  \leq C
       \Delta t^{1+\beta\rho},
\end{equation*}
and therefore, we conclude the estimate
\begin{equation} \label{I35}
 I_{3,5}  \leq C      \Delta t^{\frac{1+\beta\rho}{2}}.
\end{equation}

Now we estimate $I_{3,6}$. 
To this end, having
\begin{equation*} 
 \begin{split}
  I_{3,6}
  &\leq \sum_{j=0}^{m-1}\int_{t_j}^{t_{j+1}} 
    \|A^{\frac{\delta}{2}}S(t_m-\sigma)\|_{\cL  }\\
  &\qquad  \times\Big\|A^{-\frac{\delta}{2}} 
      F'(u(t_j))\int_0^{t_j} 
       \big(S(\sigma-\tau)-S(t_j-\tau)\big)
        \,\dd W(\tau) \Big\|_{L_2(\Omega;H)}
       \,\dd\sigma\\
  &= \sum_{j=0}^{m-1}\int_{t_j}^{t_{j+1}} 
    \|A^{\frac{\delta}{2}}S(t_m-\sigma)\|_{\cL  }\\
  &\qquad  \times\Big\|\big\|A^{-\frac{\delta}{2}} 
    F'(u(t_j))\int_0^{t_j}\int_{t_j}^\sigma 
     \dot S(\theta-\tau)\,\dd \theta\,\dd W(\tau) 
      \big\|\Big\|_{L_2(\Omega;H)}\,\dd\sigma,
 \end{split}
\end{equation*}
and using \eqref{S:1} and \eqref{F:2}, we obtain
\begin{equation*} 
 \begin{split}
  I_{3,6}
  &\leq C \big(1+\sup_{t\in[0,T]}  
  \|u(t)\|_{L_{4}(\Omega;  \dot H^{\gamma})}\big)\\
  &\quad \times\sum_{j=0}^{m-1}\int_{t_j}^{t_{j+1}} 
    (t_m-\sigma)^{-\frac{\delta\rho}{2}} 
     \Big\|\int_0^{t_j}\int_{t_j}^\sigma 
      A^{-\frac{\gamma}{2}}\dot S(\theta-\tau)
      \,\dd \theta\,\dd W(\tau) 
       \Big\|_{L_{4}(\Omega;H)}\,\dd\sigma, 
 \end{split}
\end{equation*}
Then, by \eqref{SolutionRegularity} with $p=4$ and the
Burkholder--Davis--Gundy inequality \eqref{BDG},
\begin{equation*}
 \begin{split}
  I_{3,6}
  &\leq C 
  \sum_{j=0}^{m-1}\int_{t_j}^{t_{j+1}} 
    (t_m-\sigma)^{-\frac{\delta\rho}{2}} 
\\  &\quad \times 
     \Big\|\bigg(\int_0^{t_j}\big\|\int_{t_j}^\sigma 
      A^{-\frac{\gamma}{2}}\dot S(\theta-\tau)
       \,\dd \theta\ Q^{\frac{1}{2}}\big\|_{\HS}^2
        \,\dd \tau \Big)^{\frac{1}{2}} 
         \Big\|_{L_{4}(\Omega;\IR)}\,\dd\sigma \\
  &\leq C 
     \big\|A^{\frac{\beta-\frac{1}{\rho}}{2}}
      Q^{\frac{1}{2}}\big\|_{\HS} 
       \sum_{j=0}^{m-1}\int_{t_j}^{t_{j+1}} 
        (t_m-\sigma)^{-\frac{\delta\rho}{2}}  \\
  &\quad \times 
     \Big\|\bigg(\int_0^{t_j}\Big(\int_{t_j}^\sigma 
      \|A^{\frac{-(\gamma+\beta)+\frac{1}{\rho}}{2}}
       \dot S(\theta-\tau)\|_{\cL  }
        \,\dd \theta\ \Big)^2
         \,\dd \tau \Big)^{\frac{1}{2}} 
         \Big\|, 
 \end{split}
\end{equation*}
which, using \eqref{NoiseRegularity} and \eqref{S:3}, implies
\begin{equation*} 
  I_{3,6}
 \leq C 
      \sum_{j=0}^{m-1}\int_{t_j}^{t_{j+1}} 
       (t_m-\sigma)^{-\frac{\delta\rho}{2}}
    \Big\|\bigg(\int_0^{t_j}\Big(\int_{t_j}^\sigma 
    (\theta-\tau)^{\frac{(\gamma+\beta)\rho}{2} 
     -\frac{3}{2}}
      \,\dd \theta\ \Big)^2
       \,\dd \tau \Big)^{\frac{1}{2}} 
        \Big\|. 
\end{equation*}
From this and
\begin{equation*}
 \begin{split}
  \int_0^{t_j}&\Big(\int_{t_j}^\sigma 
 (\theta-\tau)^{\frac{(\gamma+\beta)\rho}{2}
  -\frac{3}{2}}
   \,\dd \theta\ \Big)^2  \,\dd \tau  \\
  &= C\int_0^{t_j} 
   \Big((\sigma-\tau)^{\gamma\rho-\frac{1}{2}
    +\frac{(\beta-\gamma)\rho}{2}} 
     - (t_j-\tau)^{\gamma\rho-\frac{1}{2}
     +\frac{(\beta-\gamma)\rho}{2}} 
      \ \Big)^2  \,\dd \tau \\
  &= C\int_0^{t_j} 
   \Big((\sigma-\tau)^{\gamma\rho}
       (\sigma-\tau)^{-\frac{1}{2}
        +\frac{(\beta-\gamma)\rho}{2}} 
    -(t_j-\tau)^{\gamma\rho}
      (t_j-\tau)^{-\frac{1}{2}
       +\frac{(\beta-\gamma)\rho}{2}} 
        \ \Big)^2  \,\dd \tau \\
  &\leq C\int_0^{t_j} 
   \Big((\sigma-\tau)^{\gamma\rho}
         (t_j-\tau)^{-\frac{1}{2}
          +\frac{(\beta-\gamma)\rho}{2}} 
    -(t_j-\tau)^{\gamma\rho}
      (t_j-\tau)^{-\frac{1}{2}
       +\frac{(\beta-\gamma)\rho}{2}} 
        \ \Big)^2  \,\dd \tau \\
  &= C\int_0^{t_j} 
   (t_j-\tau)^{-1+(\beta-\gamma)\rho}
   \Big((\sigma-\tau)^{\gamma\rho} 
    -(t_j-\tau)^{\gamma\rho}
     \ \Big)^2  \,\dd \tau \\
  &\leq C \Delta t^{2\gamma\rho}
   \int_0^{t_j} (t_j-\tau)^{-1+(\beta-\gamma)\rho}
    \,\dd \tau \\
  &= C t_j^{(\beta-\gamma)\rho}
   \Delta t^{2\gamma\rho}, 
 \end{split}
\end{equation*}
we conclude the estimate 
\begin{equation} \label{I36}
 I_{3,6}  \leq C     \Delta t^{\gamma\rho}.
\end{equation}

To estimate $I_{3,7}$, the last term in 
\eqref{I3:s}, we have
\begin{equation*}
  I_{3,7}
 \leq  \sum_{j=0}^{m-1}
    \int_{t_j}^{t_{j+1}} 
     \big\| A^{\frac{\eta}{2}}
      S(t_m-\sigma)\|_{\cL  }
       \Big\|\big\|A^{-\frac{\eta}{2}} R_{F,j}(\sigma) 
        \big\|\Big\|_{L_2(\Omega;H)}\,\dd\sigma.
 \end{equation*}
By \eqref{S:1} and \eqref{F:3}, this implies
\begin{equation*} 
  I_{3,7}
   \leq C \sum_{j=0}^{m-1}
    \int_{t_j}^{t_{j+1}}
     (t_m-\sigma)^{-\frac{\eta\rho}{2}} 
      \|u(\sigma)-u(t_j)\|_{L_{4}(\Omega;H)}^2 \,\dd\sigma,
\end{equation*}
which, considering the fact that, 
\cite[Proposition 3.2]{AnderssonKovacsLarsson}, 
\begin{equation*} 
 \|u(\sigma)-u(t_j)\|_{L_{4}(\Omega;H)}
 \leq C (\sigma-t_j)^{\frac{\gamma\rho}{2}},
\end{equation*}
we conclude the estimate 
\begin{equation} \label{I37}
  I_{3,7}   \leq C \Delta t^{\gamma\rho}.
\end{equation}

Finally, inserting \eqref{I1}--\eqref{I4} and
\eqref{I31}--\eqref{I37} into \eqref{I3:s}, we have
\begin{equation*} 
 \begin{split}
  \|u(t_m)-U_m^N\|_{L_2(\Omega;H)}
\le C 
    \big(\Delta t^{\gamma\rho}
     + \lambda_{N+1}^{-\frac{\gamma}{2}} \big) 
     + C \Delta t \sum_{j=0}^{m-1}
    \|u(t_j)-U_j\|_{L_2(\Omega;H)},
 \end{split}
\end{equation*}
which, by the discrete Gronwall lemma, completes the proof.

\end{proof}

\begin{rem} \label{Rem:3} We note that the temporal strong rate of
  convergence is (almost) twice the rate of the backward Euler method
  combined with the first order Lubich convolution quadrature used in
  \cite{AnderssonKovacsLarsson,KovacsPrintemsStrong}. In particular,
  when $Q$ is of trace class we almost recover the deterministic order
  $O(\Delta t)$ in time (c.f.\ Remark~\ref{Rem:4}).
\end{rem}

\begin{rem}\label{Rem:4}
For the deterministic form of the model problem 
\eqref{VolterraEq}, i.e., with $W=0$, the rate is 
therefore 
$O\big(\Delta t + \lambda_{N+1}^{-\gamma} \big))$, 
as expected. 
Indeed, recalling \eqref{I1} and 
Remark~\ref{DeterministicSolutionRegularity}, 
we have 
\begin{equation*} 
  I_1
\leq \|S(t_m)\|_{\cL}
   \|A^{-\gamma}(I-\cP_N) 
    A^{\gamma}u_0\| \\ 
\leq C \lambda_{N+1}^{-\gamma} \|u_0\|_{2\gamma}.
\end{equation*}
We also recall \eqref{I2}, for which we have, in this case,
\begin{equation*}
 I_2\leq C \lambda_{N+1}^{-\gamma}(1+\|u_0\|). 
\end{equation*}
\end{rem}

\begin{rem}\label{Rem:5}
  To avoid the restrictive assumption that one has access to the
  eigenvalues and eigenfunctions of $A$, in theory, one may discretize
  equation \eqref{VolterraEq} in space by other means, such as the
  finite element method. Indeed, when $A$ is minus the Dirichlet
  Laplacian in $L_2(\cD)$, then one has nonsmooth data error estimates
  for the finite element method, at least for the main example
  \eqref{Volterrafrac}, see \cite{LST}, and the error analysis in the
  present paper can be performed with a slight increase in
  technicality using these nonsmooth data estimates. However, the
  corresponding algorithm would be difficult to implement in
  practice. Indeed, if $S_h$ denotes the finite element approximation
  of $S$, where $h$ is the finite element mesh-size, one would have to
  simulate a Gaussian random variable with covariance operator
\begin{equation*}
\int_0^{t_n}S_h(t)P_hQP_hS_h(t)\, \dd t  
\end{equation*}
which, even in the simplest case $Q=I$, is not practically feasible
unless one has access to the eigenvalues and eigenfunctions of the
discrete Laplacian $A_h$. Therefore, we have chosen to analyze the
spectral Galerkin method for which the scheme is easily implementable, but
even in this case, this is only true when $A$ and $Q$ commute.
\end{rem}

\begin{rem} \label{Rem:6} The feasibility of the proposed numerical
  scheme relies heavily on whether one knows the scalar functions
  $s_k$ from \eqref{sk}. This is the case for the Riesz kernel, see
  Section~\ref{Sec:Numeric}, or for the tempered Riesz kernel, but
  in general this leads to the additional difficulty of solving
  \eqref{sk}, for example, by numerically inverting a Laplace
  transform.
\end{rem}

\section{Numerical implementation}
 \label{Sec:Numeric}
 In this section we present the explicit form of the solution of
 \eqref{sk} in terms of the Mittag-Leffler functions.  Then,
 we illustrate the temporal strong order of convergence, to confirm
 the proposed rate in Theorem~\ref{thm:main}.

\subsection{Explicit representation of the solution}
First, we derive an explicit representation of the 
resolvent family in terms of the Mittag-Leffler functions 
when $b$ is the Riesz kernel.

Recall that the one parameter Mittag-Leffler function 
$E_{a}(z)$, $a>0$, 
is defined as
\begin{equation*}
 E_{a}(z)=\sum_{k=0}^\infty
  \frac{z^k}{\Gamma(a k+1)}, 
  \quad  z\in \mathbb{C},
\end{equation*}
where obviously $E_{1}(z)=\exp(z)$. 
By taking the Laplace transform of 
\eqref{sk}, when 
 $b(t)=t^{\rho-2}/\Gamma(\rho-1)$, we have
 $\widehat{b}(z)=z^{1-\rho}$ and 
\begin{equation*}
  \widehat{s_k}(z)= \frac{1}{z+\lambda_kz^{1-\rho}}
  =\frac{z^{\rho-1}}{z^{\rho}+\lambda_k},
\end{equation*}
which implies 
\begin{equation*}
 s_k(t)=  E_{\rho}(-\lambda_k t^\rho) .
 \end{equation*}
Thus, the resolvent family is given by 
\begin{align*}
S(t)v=\sum_{k=1}^\infty E_{\rho}(-\lambda_k t^\rho )  
(v,\vp_k)\vp_k  .
\end{align*}

To explain the computer implementation of the 
fully-discrete method \eqref{UmN}, we note that 
\begin{equation*}
 S_N(t_m)= \sum_{k=1}^N E_{\rho}(-t_m^\rho \lambda_k)  
 (v,\vp_k)\vp_k,  
\end{equation*}
Suppose that $Q$ has the same eigenfunctions as $A$, so that
$Qv=\sum_{k=1}^\infty \mu_k (v,\vp_k)\vp_k$. 
Then, for each time step $m=1,\dots,M$, 
the approximation $U^N_m$ defined by \eqref{UmN} 
is given by $U^N_m=\sum_{k=1}^NU_{m,k}^N\vp_k$, where
for $k=1,\dots,N$, 
\begin{equation} \label{UmN1}
 \begin{split}
   U_{m,k}^N=E_{\rho}(-\lambda_k t_m^\rho)u_{0,k} 
   &+ \sum_{j=0}^{m-1}
   \int_{t_{j}}^{t_{j+1}} 
    E_{\rho}(-\lambda_k (t_m-\sigma)^\rho)
     \,\dd \sigma \ F_k(U_j^N) \\
  &+\int_{0}^{t_m}
   E_{\rho}(-\lambda_k (t_m-\sigma)^\rho)\mu_k^{\frac12}
    \,\dd \beta_k(\sigma) 
 \end{split}
\end{equation}
and where $u_{0,k}=(u_0,\vp_k),\ F_k(\cdot)=(F(\cdot),\vp_k)$, and
$\beta_k$, $k=1,\dots,N$, are mutually independent standard Brownian
motions.

We note that the integrals of the Mittag-Leffler functions are
computable, e.g., by means of a simple quadrature, say the trapezoidal
method.  For evaluating the Mittag-Leffler function we use
\texttt{mlf.m} from \cite{PodlubnyKacenak}. What one has to be careful
with is how to simulate, for fixed $k$, the random variables
$$
\mathcal{O}(t_m):=\int_{0}^{t_m}
   E_{\rho}(-\lambda_k (t_m-\sigma)^\rho)
   \mu_k^{\frac12}\, \dd \beta_k(\sigma), \quad m=1,...,M.
$$
Observe that the $\mathbb{R}^M$-valued random variable
$$
\mathcal{N}:=(\mathcal{O}(t_1),\dots,\mathcal{O}(t_M))
$$
is a $0$-mean Gaussian random variable with covariance matrix 
$$
(M)_{i,j}=\int_0^{\min(t_i,t_j)}
E_{\rho}(-\lambda_k (t_i-\sigma)^\rho)
E_{\rho}(-\lambda_k (t_j-\sigma)^\rho)\mu_k\,\dd\sigma .
$$
Thus $\mathcal{N}=L\xi$, where $LL^T=M$, and $\xi$ is an $\mathbb{R}^M$-valued
random variable with independent standard Gaussian coordinates. This
difficulty does not arise in the memoryless case as there one can
exploit the semigroup property of the solution operator. In that case
one only has to simulate the independent Gaussian random variables
$\xi_i:=\int_{t_{i-1}}^{t_i}\exp(-\lambda_k(t_i-\sigma))\mu_k^{\frac12}\,\dd
\beta_k(\sigma)$, $i=1,...,M$, and then take
$\mathcal{O}(t_m)=\sum_{i=1}^m \exp(-\lambda_k(t_m-t_i))\xi_i$,
$m=1,...,M$.

\subsection{Numerical experiments}
Since the main contribution of this paper is the temporal
approximation, we only present a simplified numerical 
experiment with uncoupled eigenmodes.  
More precisely, let 
$u=\sum_{k=1}^{\infty}u_k\vp_k$
and define the nonlinear operator 
\begin{align} \label{eq:sum}
  F(u)=\sum_{k=1}^{\infty}f(u_k)\vp_k.  
\end{align}
We simulate various coordinates of the numerical approximation; that
is, we simulate the random variables $U_{M,k}^N$ in \eqref{UmN1} for
various values of $k$, with $\mu_k=1$ and $\lambda_k=k^2\pi^2$ being
the eigenvalues of the Dirichlet Laplacian in 1-D on
$\mathcal{D}=[0,1]$.

Since we are simulating scalar problems, the noise is trace class and
we expect the rate of convergence of the MLEI to be almost 1 according
to the theory. We also compare the performance of the MLEI to the
backward Euler based convolution quadrature (BE) proposed and analyzed
in \cite{KovacsPrintemsStrong,KovacsPrintemsWeak} in the linear case
and in \cite{AnderssonKovacsLarsson} in the semilinear setting. For BE
the theory in \cite{AnderssonKovacsLarsson} predicts a strong rate of
almost 0.5 for trace class noise, albeit the conditions on $f$ there
are somewhat different to the present setting and hence the rate could
be higher for smooth additive noise. In fact, in the scalar case we
have $\|A^{t}Q^{1/2}\|_{\HS}=\lambda_k^t<\infty$ for any $t$ so that we have smoothness of any order. But when
$\lambda_k $ is large, this quantity is also large, making the error constant large, and then we do not expect
to see a higher  rate than corresponding to $t=0$;
that is, the trace class noise case (this is seen in Figure~\ref{fig2}). This
also explains why, for smaller $\lambda_k $, we might occasionally
observe a better rate than 0.5 for BE (see Figure~\ref{fig1} with
$\lambda_2$). Nevertheless, in all experiments the MLEI outperforms BE
by far and we experimentally see rate 1 for MLEI in all 
experiments.

We use $100$ sample paths in all experiments.  
The computed solution is compared to a reference solution with 
much smaller mesh size.  
We use the functions $f(u)=\sin(u)$ and $f(u)=5(1-u)/(1+u^2)$ 
and different values $\rho=1.2$ and $\rho=1.75$ 
(and for Figure~\ref{fig3}
also $\rho=1.5$); the first being closer to the heat equation 
with solutions dying out quickly (Figure~\ref{fig4}), while 
the latter producing more pronounced wave phenomenon 
(Figure~\ref{fig5}).
For functions with larger Lipschitz constant, 
such as, for example, $f(u)=100(1-u)/(1+u^2)$, we still get similar convergence behaviour as shown in Figure~\ref{fig3}.

The captions in and underneath the figures explain the settings for the
various experiments.

\begin{figure}[htbp]  
 \begin{center}
  \includegraphics[width=6.25cm,height=4.45cm]
    {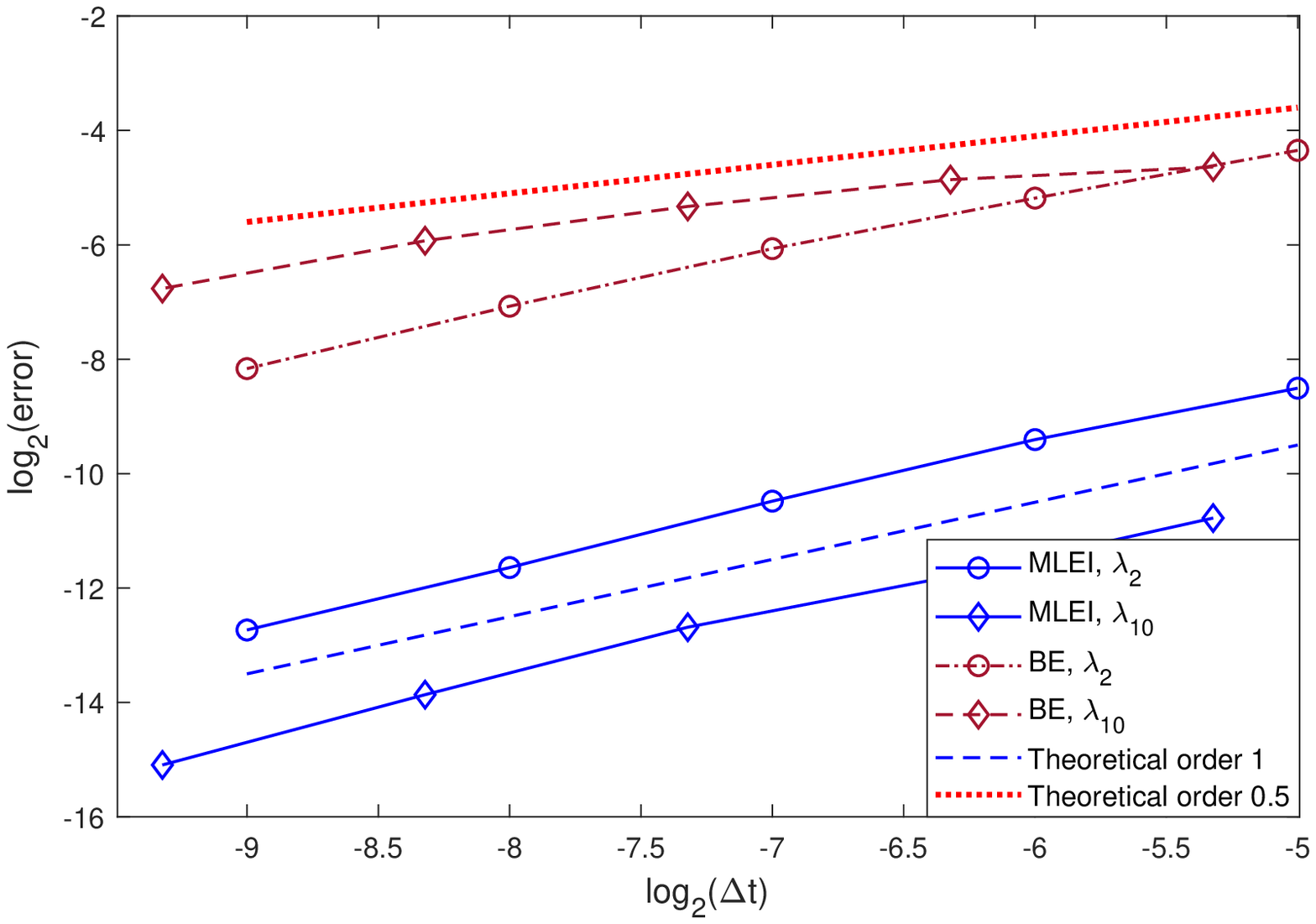}
  \includegraphics[width=6.25cm,height=4.45cm]
    {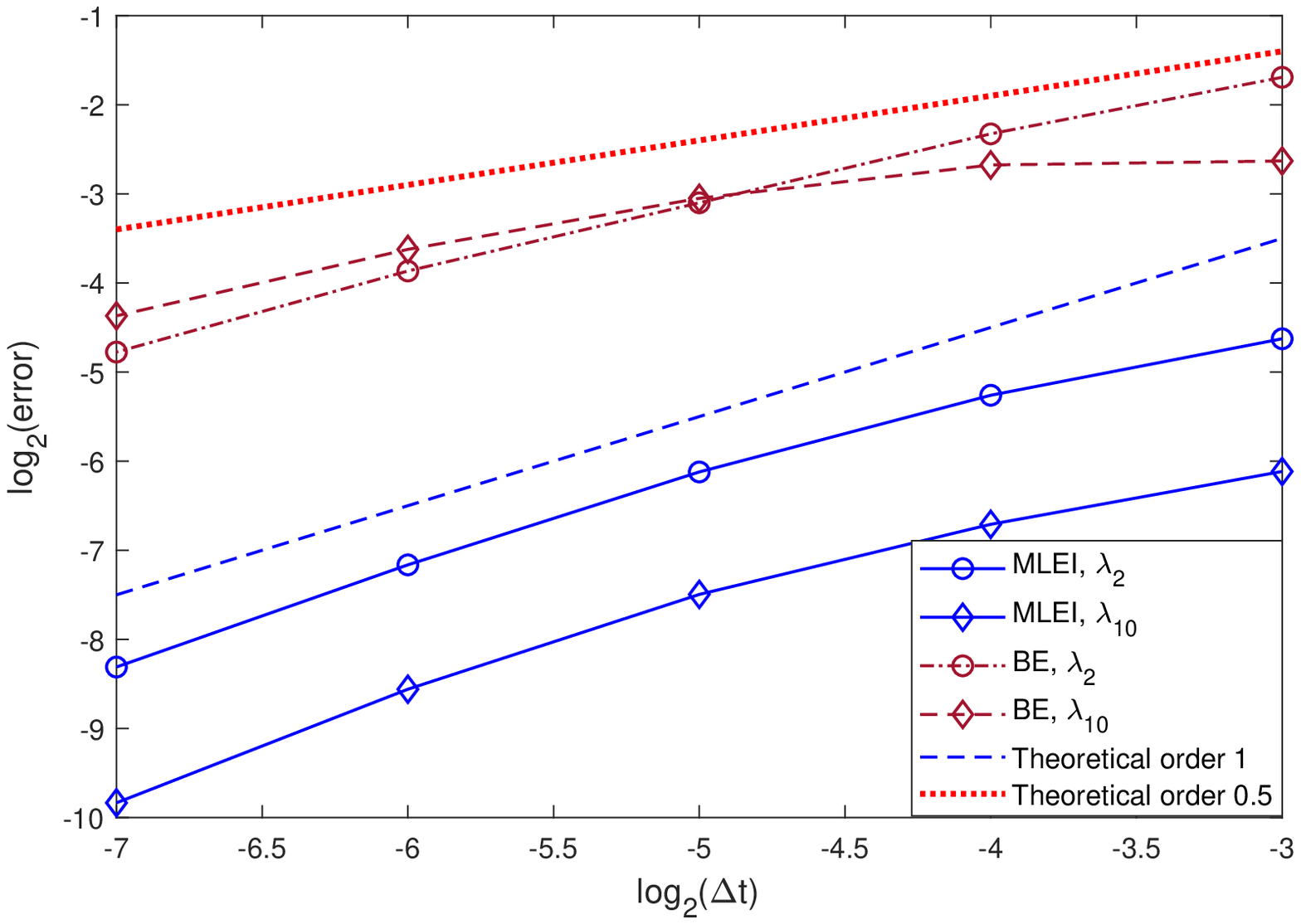}
 \caption{Comparison of BE and MLEI  temporal rate of convergence with
   $f(u)=\sin(u)$ and $\lambda_2=4\pi^2 $, respectively
   $\lambda_{10}=100\pi^2 $. 
   Left: with $\rho=1.2$. Right: $\rho=1.75$. }
  \label{fig1}
 \end{center}
\end{figure}

\begin{figure}[htbp] 
 \begin{center}
  \includegraphics[width=10.0cm,height=5.5cm]
    {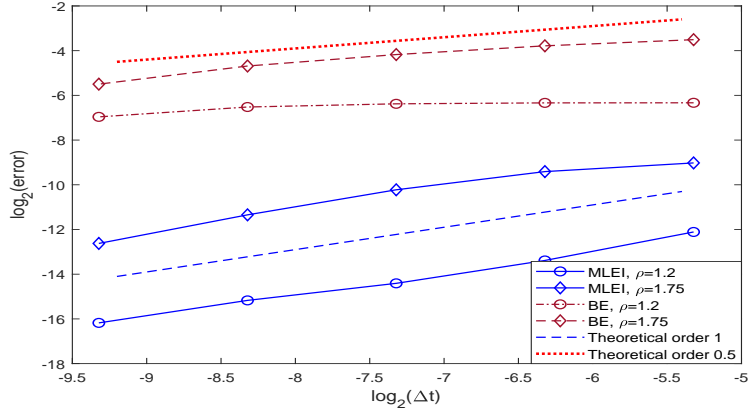}
  \caption{Comparison of BE and MLEI  temporal rate of convergence
    with $f(u)=\sin(u)$, $\rho=1.2$ and $\rho=1.75$ and
    $\lambda_{30}=900\pi^2 $.} 
 \label{fig2}
 \end{center}
\end{figure}

\begin{figure}[htbp] 
 \begin{center}
  \includegraphics[width=10.0cm,height=5.5cm]
    {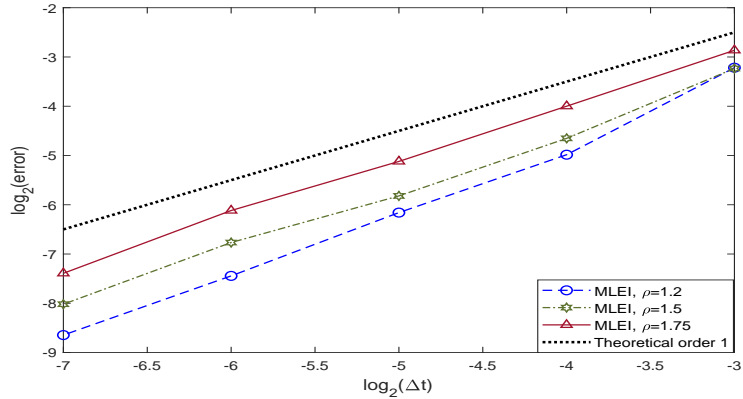}
  \caption{Temporal rate of convergence for MLEI  
  with $f(u)= 5(1-u)/(1+u^2) $, $\rho=1.2,1.5$ and $1.75$ and $\lambda_2=4\pi^2 $.}
 \label{fig3}
 \end{center}
\end{figure}

\begin{figure}[htbp]  
 \begin{center}
  \includegraphics[width=11.5cm,height=5.0cm]
    {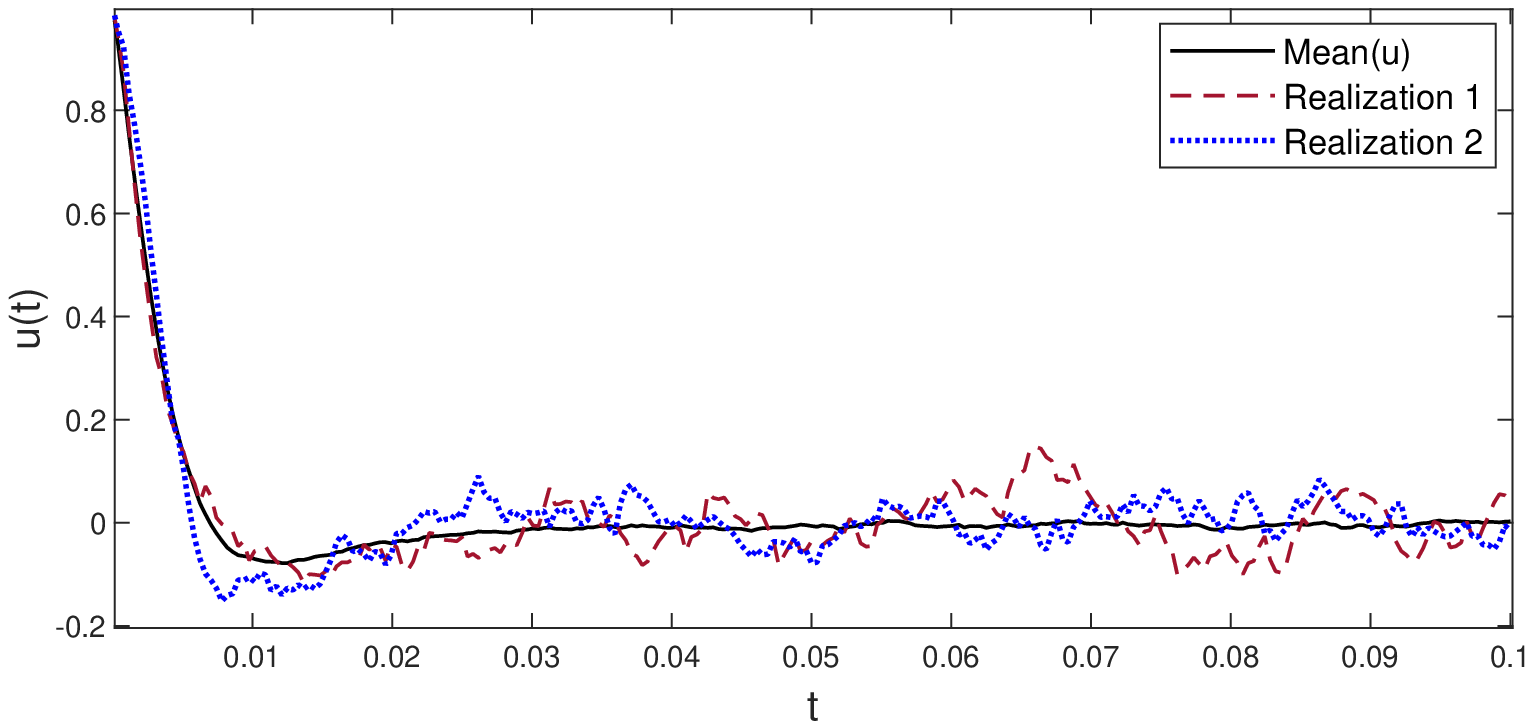}
    \caption{Solution behaviour with: $f(u)=\sin(u)$, $\rho=1.2$ and 
     $\lambda_{10}=100\pi^2 $.}
    \label{fig4}
 \end{center}
\end{figure}

\begin{figure}[htbp]  
 \begin{center}
  \includegraphics[width=11.5cm,height=5.0cm]
    {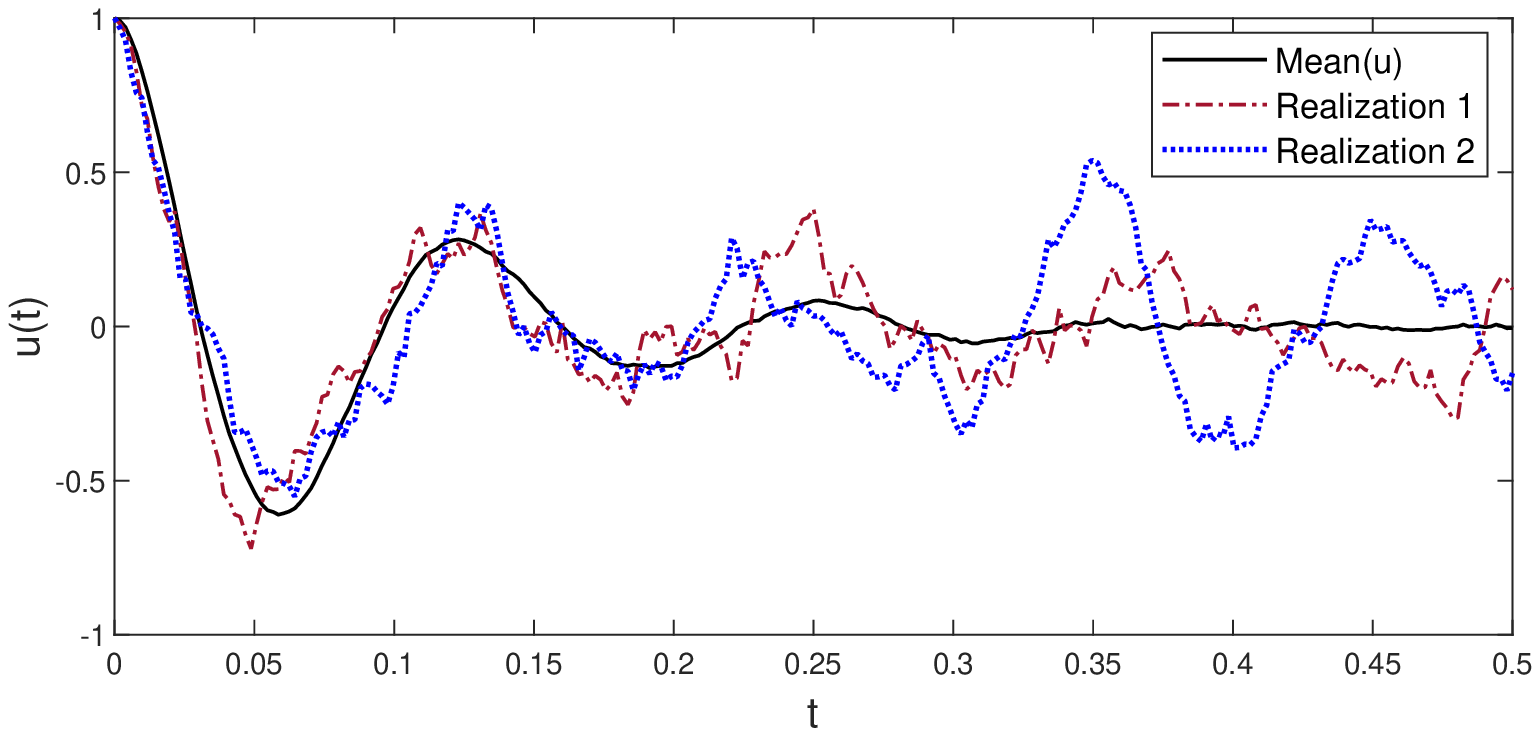}
  \caption{Solution behaviour with: $f(u)=\sin(u)$, $\rho=1.75$ and 
    $\lambda_{10}=100\pi^2 $.}
  \label{fig5}
 \end{center}
\end{figure}

\textbf{Acknowledgment.} 
We would like thank the anonymous referees for
constructive comments that helped us improve the manuscript.


\begin{thebibliography}{9}
\bibitem{AnderssonKovacsLarsson}
A.~Andersson, M~Kov\'acs, and S.~Larsson,  
Weak error analysis for semilinear stochastic 
Volterra equations with additive noise. 
J. Math. Anal. Appl. 437 (2016), 1283--1304.
\bibitem{AntonCohenLarssonWang}
R.~Anton, D.~Cohen, S.~Larsson, and X.~Wang,  
Full discretization of semi-linear stochastic 
wave equations driven by  multiplicative noise. 
SIAM J. Numer. Anal. 54 (2016), 1093--1119.
\bibitem{BaeumerGeissertKovacs}
B.~Baeumer, M.~Geissert, and M.~Kov\'acs,  
Existence, uniqueness and regularity for a class of 
semilinear stochastic Volterra equations with 
multiplicative noise. 
J. Differential Equations 258 (2015), 535--554.
\bibitem{BeckerJentzenKloeden}
S.~Becker, A.~Jentzen, and P.~Kloeden,  
An exponential Wagner-Platen type scheme for SPDES. 
SIAM J. Numer. Anal. 54 (2016), 2389--2426.
\bibitem{DaPratoZabczyk}
G.~Da Prato and J.~Zabczyk, 
Stochastic Equations in Infinite Dimensions, 
Encyclopedia Math. Appl., vol. 44, Cambridge University Press, Cambridge, 1992. 

\bibitem{Gunz}
  M.~Gunzburger, B.~Li, and J.~Wang, Sharp convergence
  rates of time discretization for stochastic time-fractional PDEs
  subject to additive space-time white noise.  Math. Comp.  88 (2019),
  1715--1741.
 
\bibitem{Gunz1}
  M.~Gunzburger, B.~Li, and J.~Wang, Convergence of
  finite element solutions of stochastic partial integro-differential
  equations driven by white noise. Numer. Math. 141 (2019),
  1043--1077.
\bibitem{HochbruckOstermann} M.~Hochbruck and A.~Ostermann,
  Exponential integrators.  Acta Numerica (2010), 209--286.
\bibitem{JentzenKloedenWinkel}
  A.~Jentzen, P.~Kloeden, and G.~Winkel,
  Efficient simulation of nonlinear parabolic SPDEs with additive
  noise.  Ann. Appl. Probab. 21 (2011), 908--950.
\bibitem{KovacsLarssonLindgren} M.~Kov\'acs, S.~Larsson, and
  F.~Lindgren, Weak convergence of finite element approximations of
  linear stochastic evolution equations with additive noise II: Fully
  discrete schemes.  BIT Numer. Math. 53 (2013), 497--525.
\bibitem{KovacsPrintemsStrong} M.~Kov\'acs and J.~Printems, Strong
  order convergence of a fully discrete approximation of a linear
  stochastic Volterra equation.  Math. Comp. 83 (2014), 2325--2346.
\bibitem{KovacsPrintemsWeak} M.~Kov\'acs and J.~Printems, Weak
  convergence of a fully discrete approximation of a linear stochastic
  evolution equation with a positive-type memory term.
  J. Math. Anal. Appl. 413 (2014), 939--952.
\bibitem{LarssonMolteni} S.~Larsson and M.~Molteni, Numerical solution
  of parabolic problems based on a weak space-time formulation.
  Comput. Methods Appl. Math. 17 (2017), 65--84.
\bibitem{LordTambue} G.~J.~Lord and A.~Tambue, Stochastic exponential
  integrators for the finite element discretization of SPDEs for
  multiplicative and additive noise.  IMA J. Numer. Anal. 33 (2013),
  515--543.
  
\bibitem{LST} Ch.~Lubich, I.~Sloan, and V.~Thom\'ee, Nonsmooth data
  error estimates for approximations of an evolution equation with a
  positive-type memory term. Math. Comp. 65 (1996), 1--17.

\bibitem{Mainardi} F.~Mainardi and P.~Paradisi, Fractional diffusive
  waves.  J. Comput. Acoust. 9 (2001), 1417--1436.
\bibitem{McLeanThomee} W.~McLean and V.~Thom\'ee, Numerical solution
  of an evolution equation with positive memory term.
  J. Austral. Math. Soc. Ser. B 35 (1993), 23--70.
\bibitem{PodlubnyBook1999} I.~Podlubny, Fractional Differential
  Equations, Academic Press, San Diego, 1999.
\bibitem{PodlubnyKacenak} I.~Podlubny and M.~Kacenak, The Matlab mlf
  code. MATLAB central File Exchange (2001-2009). File ID: 8738.
\bibitem{PrussBook1993} J.~Pr\"uss, Evolutionary Integral Equations
  and Applications.  Monographs in Mathematics, vol. 87, Birkh\"auser,
  Basel, 1993.
\bibitem{Wang}
X.~Wang, 
Strong convergence rates of the linear implicit 
Euler method for the finite element discretization 
of SPDEs with additive noise. 
IMA J. Numer. Anal. 37 (2017), 956--984. 
\bibitem{WangGanTang}
X.~Wang, S.~Gan, and J.~Tang, 
Higher order strong approximations of semilinear
stochastic wave equation with additive space-time
white noise. 
SIAM J. Sci. Comput. 36 (2014), A2611--A2632. 

\bibitem{WangQi}
X.~Wang and R.~Qi,  
A note on an accelerated exponential Euler 
method for parabolic SPDEs with additive noise.
Appl. Math. Letters  46 (2015), 31--37.

\end{thebibliography}

\end{document}